\documentclass[10pt,a4paper]{article}
\usepackage{amssymb,comment,amsmath,amsthm}
\usepackage{graphicx}
\usepackage{color}
\usepackage{titling,titlesec}
\usepackage{url}
\usepackage{amsthm,lipsum}
\usepackage{geometry}
\usepackage{tikz}
\usepackage{hyperref}
\titleformat{\subsection}[runin]
{\normalfont\bfseries}{}{0em}{}[.]
\newcounter{foo}

\newfont{\blb}{msbm10 scaled\magstep1}
\newfont{\comp}{cmr12 scaled\magstep1}
\newfont{\compb}{cmr10 scaled\magstep2}
\newfont{\sbb}{cmssbx10 scaled\magstep3}
\newfont{\sbbb}{cmssbx10 scaled\magstep5}
\newfont{\sbs}{cmssbx10 scaled\magstep1}
\newtheorem{theorem}{Theorem}
\newtheorem{lemma}[subsection]{Lemma}
\newtheorem{claim}[subsection]{Claim}

\newtheorem{corollary}[subsection]{Corollary}
\newtheorem{proposition}{Proposition}
\newtheorem{definition}[foo]{Definition}
\newtheorem{conjecture}{Conjecture}

\newcommand{\cev}[1]{\reflectbox{\ensuremath{\vec{\reflectbox{\ensuremath{#1}}}}}}

\newcommand{\im}{\mathrm{Im}}
\newcommand{\mon}{\mathrm{Inj}}
\newcommand{\Hom}{\mathrm{Hom}}
\newcommand{\Mon}{\mathrm{Inj}}
\newcommand{\vep}{\varepsilon}

\allowdisplaybreaks
\title{A Tight Lower Bound on Trees in Graphs}
\author{
Chase Wilson\thanks{Department of Mathematics, University of California, San Diego, CA, 92093-0112 USA. E-mail: c7wilson@ucsd.edu}}
\begin{document}
\date{ }
\setlength{\droptitle}{-2.5cm}
\maketitle
\vspace{-0.5in}
\begin{abstract}
Mubayi and Verstraete conjectured that if $T$ is a tree on $t + 1$ vertices, then any $n$-vertex graph $G$ with average degree $d$ contains at least 
\[
n d(d - 1) \cdots (d - t + 1)
\]
labeled copies of $T$ as long as $d$ is sufficiently large compared to $t$. We prove this is true and show that when the diameter of $T$ is at least $3$, equality holds iff $G$ is the disjoint union of cliques of size $d + 1$. When the diameter is $2$, equality holds iff $G$ is $d$-regular.
\end{abstract}
\section{Introduction}
%In 1982, Erdos and Simonovits showed that if $G$ is an $n$-vertex graph of average degree $d$, then the number of walks of length $t$ in $G$, $|\Hom(P_t, G)|$ is at least $n d^t$ as long as $t$ is even. They also proved that the number of paths (i.e. injective walks) of length $t$ in $G$ is at least $(1 - o(1)) |\Hom(P_t, G)|$.
Let $G$ and $H$ be graphs. A {\it graph homomorphism} from $H$ to $G$ is a function $\varphi : V(H) \rightarrow V(G)$ such that $\{\varphi(u), \varphi(v) \} \in E(G)$ whenever $\{u, v\} \in E(H)$ -- in other words, an edge-preserving map. When $\varphi$ is an injection, then $\varphi$ is called an {\em injective homomorphism}.
We let $\Hom(H, G)$ and $\Mon(H, G)$ denote the set of homomorphisms and injective homomorphisms from $H$ to $G$ respectively; $\Mon(H, G)$ is precisely the number of labeled
subgraphs of $G$ which are isomorphic to $H$. The studies of $|\Hom(H,G)|$ and $|\Mon(H,G)|$ have a long research history, including celebrated problems such as the Erd\H{o}s-Rademacher problem on the minimum number of triangles in an $n$-vertex graph with $m$ edges, solved by Liu, Pikhurko and Staden~\cite{LPS}. The 
remarkable Sidorenko conjecture~\cite{Sidorenko} and related Erd\H{o}s-Simonovits conjecture~\cite{ES} that for any bipartite graph $H$,
if $G_n$ is any $n$-vertex graph of density at least $p \in [0,1]$, then
\begin{equation}\label{sidconj} \liminf_{n \rightarrow \infty} \frac{|\Hom(H,G)|}{n^{|V(H)|}} \geq p^{|E(H)|}.
\end{equation}
For progress on this conjecture, see Conlon, Lee and Sidorenko~\cite{CLS} and the references therein. 

\medskip

A special case is the problem of counting walks of a specific length: if $P_t$ denotes a path with $t$ edges, then the number of walks of length $t$ in $G$ is precisely $|\Hom(P_t,G)|$. This problem has a long history, going back to early results of Atkinson, Moran and Watteson~\cite{AMW}, Mulholland and Smith~\cite{MS} and Blakley and Roy~\cite{BR}. A {\em walk of length $t$} in a graph $G$ is an alternating sequence $(v_1,e_1,v_2,\dots,v_t,e_k,v_{t + 1})$ where $v_i \in V(G)$ and $e_i = \{v_i,v_{i + 1}\}$ for $1 \leq i \leq t$. Note that such a walk allows repeated vertices $v_i$ and repeated edges $e_i$, and if $P_t$ denotes a path with $t$ edges, then the number of walks of length $t$ in $G$ is precisely $|\Hom(P_t,G)|$.  It is not hard to see that if $G$ is an $n$-vertex $d$-regular graph, then the number of walks of length $t$ in $G$ is exactly $nd^t$. If $G$ is a graph of average degree $d$, and $t \geq 2$ is even, then using spectral methods it is 
straightforward to see that the number of walks of length $t$ in $G$ is at least $nd^t$. Mulholland and Smith~\cite{MS} and, independently, Blakley and Roy~\cite{BR}, were the first to show that the same conclusion holds when $t$ is odd, extending the earlier work of Atkinson, Moran and Watteson~\cite{AMW} for $t = 3$. They established {\em H\"{o}lder's inequality} for entrywise non-negative symmetric matrices $A$, namely, if $z \in \mathbb R^n$ is any entrywise non-negative vector, and $A$ is any entrywise non-negative symmetric $n$ by $n$ matrix, then $\left<A^t z,z\right> \geq \left<Az,z\right>^t$ -- see also London~\cite{London} for a generalization. A relatively short proof of this inequality is available via the so-called {\em tensor trick} (see~\cite{MVbook} page 92). 
Erd\H{o}s and Simonovits~\cite{ES2} conjectured further that $\left<A^tz,z\right>^{1/t}$ 
is non-decreasing over odd values of $t$, and this was proved by Sa\u{g}lam. Subsequently, a short proof was found by Blekherman and Raymond~\cite{BRa}, 
and further generalizations were investigated by T\"{a}ubig and Weihmann~\cite{TW}. 

\medskip

A {\em non-backtracking walk} is a walk $(v_1,e_1,v_2,\dots,v_t,e_t,v_{t + 1})$ such that $e_i \neq e_{i + 1}$ for $1 \leq i \leq t$ -- in other words, 
the walk does not traverse an edge twice in consecutive steps. The work of Hoory~\cite{Hoory} and Alon, Hoory and Linial~\cite{AHL} uses an entropy based approach to 
prove the above-mentioned results of Blakley-Roy and Mulholland-Smith on counting walks of length $t$, and further showed that the number of non-backtracking walks of 
length $t$ in an $n$-vertex graph $G$ of average degree $d \geq 2$ is at least $nd(d - 1)^{t-1}$ -- indeed they constructed a random greedy homomorphism and showed that for every edge $e \in E(T)$, the marginal distribution of the image of $e$ is a uniform random edge in $G$. Erd\H{o}s and Simonovits~\cite{ES2} showed that 
the number of paths of length $t$ in $G$ -- in other words $|\Mon(P_t,G)|$ -- is at least $nd^t - O(nd^t/\log\log d)$ as $d \rightarrow \infty$. 
Sidorenko proved (\ref{sidconj}) for a variety of graphs, including trees -- see~\cite{CLS,Sidorenko2}. Mubayi and Verstra\"{e}te~\cite{MV} gave a stronger conclusion, namely for $d \geq t$ and any tree $T$ with $t$ edges:  
\[ |\Mon(T,G)| \geq n(1 - t^5/d^2)d(d - 1)\dots(d - t + 1).\]
This strengthens results of Dellamonica et al~\cite{DH}.
Note that the number of copies of $T$ in a disjoint union of complete graphs $K_{d + 1}$ is precisely $nd(d - 1) \dots (d - t + 1)$. It was conjectured in~\cite{MV}
that the term $1 - t^5/d^2$ can be removed if $d$ is large enough relative to $t$, to obtain a tight bound on the number of copies of $T$ in any $n$-vertex 
graph $G$ of average degree $d$. In this paper, we prove this conjecture:

\begin{theorem} {\label{main}}
For all integers $t > 0$, there exists a $d_0(t)$ such that if $T$ is a $t + 1$-vertex tree and $G$ is an $n$-vertex graph with average degree $d \geq d_0(t)$, then
\begin{equation} {\label{mainbound}}
	|\Mon(T, G)| \geq n d(d - 1) \cdots (d - t + 1).
\end{equation}
If $T$ has diameter at least 3, then equality holds iff $G$ is the disjoint union of cliques of size $d + 1$.
\end{theorem}
 
The value of $d_0(t)$ obtained from our proof of Theorem \ref{main} is polynomial in $t$ -- of order at most $t^4$.
 It may be that the above theorem holds with $d_0(t) = t$. This would imply the famous Erd\H{o}s-S\'{o}s Conjecture~\cite{ErdosSos}, 
which states that every graph of average degree at least $t$ contains every tree with $t$ edges. There are many other applications that do not require $d_0(t) = t$. For example, counting paths in graphs comes up frequently in contexts such as bounding the spectral radius or energy of a graph or counting the number of words with a forbidden pattern \cite{MV}. We make the following conjecture, which remains open for paths of length $t \geq 4$.

\begin{conjecture}
Let $T$ be a tree with $t$ edges, and let $G$ be an $n$-vertex graph with average degree $d \geq t - 1$. Then $|\Hom(T,G)| \geq n(d)_t$. 
\end{conjecture}

It is easy to prove this conjecture when $T$ has diameter $2$, as $T$ is a star with $t$ leaves, and for a graph $G$ of average degree $d \geq t$, we may apply Jensen's inequality to 
\[
|\Mon(T, G) | = \sum_{v \in V(G)}  (d(v))(d(v) - 1) \cdots (d(v) - t + 1).
\]

Instead of proving Theorem \ref{main} directly, the focus of the paper will be on proving the following "weaker" lemma, from which the full Theorem \ref{main} easily follows [see appendix]

\begin{lemma}{\label{weak}}
For all integers $t > 0$, there exists a $d_0(t)$ such that if $T$ is a $t + 1$-vertex tree and $G$ is an $n$-vertex graph with average degree $d \geq d_0(t)$ and minimum degree $\delta(G) \geq d/4$, then
\begin{equation}
	|\Mon(T, G)| \geq n d(d - 1) \cdots (d - t + 1).
\end{equation}
If $T$ has diameter at least 3, then equality holds iff $G$ is the disjoint union of cliques of size $d + 1$.
\end{lemma}

{\bf Organization.} We prove Lemma \ref{weak} by bounding the entropy of a random ``greedy'' embedding of $T$ into $G$ constructed in~\cite{MV}. Analogous entropy arguments have been used to get tight lower bounds for $|\Hom(T, G)|$ and locally injective homomorphisms \cite{AHL}. In these cases, the marginal distribution at each edge of $T$ is a uniform random edge in $G$ and the conclusion is quickly established using Jensen's Inequality. Mubayi and Verstra\"ete~\cite{MV} proved their bound for the injective case by showing the embedding is close to uniform at every edge. We show in subsequent section that the negative error from this is only substantial when $G$ is far from a disjoint union of cliques in which case the other naive estimates are far enough from equality to cancel it out.

\section{Random embedding and entropy} 

In accordance with Lemma \ref{weak}, we will assume that $\delta(G) \geq d/4$. We fix a breadth-first search ordering $(x_0, \cdots, x_t)$ of the vertices in the tree $T$, where $x_0$ is a leaf. We let $T^i$ be the subtree of $T$ on the vertex set $\{x_0, \cdots, x_i\}$. We let $a(i)$ be the index of the anscestor of $x_i$, that is the unique index $j$ for which $x_{j}$ is the parent of $x_i$. For $\gamma \in \mon (T, G)$, we write $\gamma = (\gamma_0, \cdots, \gamma_t)$ to mean $\gamma(x_j) = \gamma_j$. We also let $\gamma^i$ be the restriction of $\gamma$ to $T^i$ so $\gamma^i = (\gamma_0, \cdots, \gamma_i)$. We will choose a random greedy embedding $\phi \in \mon(T,G)$ of $T$ in $G$ as follows according to the following distribution: For $v \in V(G)$, we let $\phi_0 = v$ with probability $d(v)/(2|E|)$. Then for $0 < i \leq t$, $\phi_i = \phi(x_i)$ is chosen uniformly over
\[
N_+ (i) = N(\phi_{a(i)}) - \{\phi_0, \cdots, \phi_{i - 1} \}
\]

\subsection{Entropy}

Our proof of Theorem \ref{main} uses entropy methods. We briefly describe entropy and state its properties that we will use.
If $S$ is a finite set and $X, Y$ are $S$-valued random variable, then the {\em entropy of $X$} is given by
\[
	H[X] = \sum_{y \in \im(X)} P(X = y) \log \left ( \frac{1}{ P(X = y)} \right )
\]
and the {\em conditional entropy} $H[X | Y]$ is given by
\[
	H[X | Y] = E_{y \sim Y} [ H[(X | Y = y)]]
\]
Two of the basic properties of entropy are as follows:

\begin{proposition} 
Let $X,X_1,X_2,\dots,X_k$ be random variables where $X$ has values in $S$ and $X_i$ has values in $S_i$ for $1 \leq i \leq k$. 
\begin{center}
\begin{tabular}{lp{5in}}
$1$. & {\rm [Uniform bound]}  $H[X] \leq \log( |\im(X)| )$ with equality iff $X$ has uniform distribution.\\
$2$. & {\rm [The Chain Rule]} $(X_1,X_2, \cdots, X_k)$ is an $S_1 \times \cdots \times S_k$ random variable and $H[X_1, \cdots, X_k] = H[X_1] + H[X_2 | X_1] + \cdots H[X_k | X_1, \cdots, X_{k - 1} ]$.
 \end{tabular}
 \end{center}
\end{proposition}

\subsection{The Model Proof} Let us begin by showing how to prove Theorem \ref{main} under the assumption that $P(\phi_i = v) = \deg(v)/2|E(G)|$ for every $v \in V(G)$, which is true in graphs which have e.g. girth at least $t$. In general, however, this is false, but as we will see the quantities $P(\phi_i = v)$ and $\deg(v)/2|E(G)|$ are very close together and the real proof of Theorem \ref{main} is an adaptation of the following proof to deal with this small error.

\medskip 

We compute the entropy of $\phi$. By the chain rule,
\[
H[ \phi ] = H[\phi^1] + \sum_{i = 1}^{t - 1} H[\phi^{i + 1} | \phi^i]
\]
Since $\phi^1$ is uniform over all edges, $H[\phi^1] = \log(2 |E(G)|) = \log(nd)$ and 
\[
H[\phi^{i + 1} | \phi^i] = \sum_{v \in V} \frac{d(v)}{2 |E(G)|} \log ( n_+(v) ) \geq \sum_{v \in V} \frac{d(v)}{2 |E(G)|} \log ( d(v) - i ) \geq \log(d - i)
\]
where the second inequality is Jensen's Inequality aplied to the function $x \log(x - i)$. This gives $ \log( |\mon(T, G) | ) \geq H[ \phi ] \geq \log(n d(d - 1) \cdots (d - t + 1))$, and exponentiating proves Theorem \ref{main}.

\section{Notation and Setup}

We will be constructing an embedding $\gamma \in \mon (T^i, G)$ one vertex at a time and when we add the $i$th vertex, we will be thinking about $T^{i - 1}$ as a path from $x_0$ to $x_{a(i)}$ with added branches that are a nuisance to deal with. In light of this, for vertices $u, v \in V(G)$ and path $p = (p_0, \cdots, p_k)$ in $G$, we let
\begin{align*}
	\Gamma_i (u, *) & = \{ \gamma \in \mon (T^i, G) : \gamma(x_0) = u \}\\
	\Gamma_i (*, v) & = \{ \gamma \in \mon (T^i, G) : \gamma(x_{a(i)}) = v \}\\
	\Gamma_i (u, v) & = \{ \gamma \in \mon (T^i, G) : \gamma(x_0) = u, \gamma(x_{a(i)}) = v \}\\
	\Gamma_i (p) & = \{ \gamma \in \mon (T^i, G) : \gamma(x_{a(i + 1)}) = p_k, \gamma(x_{a(a(i + 1))}) = p_{k - 1}, \cdots, \gamma(x_{a^{k + 1}(i + 1)})= p_0, a^{k + 1}(i + 1) = 0 \}
\end{align*}

For a path $p = (p_0, \cdots, p_k)$, we let $\cev{p}$ denote the reverse path $\cev{p} = (p_k, \cdots, p_0)$.

\section{Bounding the Entropy of $\phi$}
Our goal will be to show $H[\phi] \geq \log(n (d)_t)$ when $d \geq d_0$. Just like in the model proof we begin with the chain rule,
\begin{equation*}
H[\phi] = H[\phi^1] + \sum_{i = 1}^{t - 1} H[\phi^{i + 1} | \phi^i]
\end{equation*}
and we note that distribution of $\phi^1$ is uniform over all oriented edges and hence
\[
H[\phi^1] = \log(2 |E(G)|) = \log(n d)
\]
Next, since $P \left (\phi^i \in \gamma^i(v, \cdot)  \right ) = P(\phi_0 = v) = d(v)/2|E|$
\begin{align*}
	H[\phi^{i + 1} | \phi^i ] & =\sum_{\gamma \in \Gamma_i} P \left (\phi^i = \gamma \right ) \log (d(\gamma_{a(i + 1)}) - |N(\gamma_{a(i + 1)} ) \cap \{ \gamma_0, \cdots, \gamma_i \} |)\\
	&= \sum_{v \in V} \left ( \frac{d(v)}{2|E|} - P \left  (\phi^i \in \Gamma(v, *) \right ) + P \left ( \phi^i \in \Gamma( *, v) \right ) \right ) r_i(v)\\
		& = \log( d - i) + \Pi_i^1 + \Pi_i^2 - \Pi_i^3
\end{align*}
where
\begin{align*}
	\Pi_i^1 &= \sum_{v \in V} \left [ \frac{d(v)}{2|E|} \log( d(v) - i ) \right ] - \log(d - i) \\
	\Pi_i^2 &= \sum_{v \in V} \frac{d(v)}{2|E|} \left [ r_i(v) - \log (d(v) - i) \right ]\\
	\Pi_i^3 &= \sum_{v \in V} \left [ P \left (\phi^i \in \Gamma(v, *) \right ) - P(\phi^i \in \Gamma(*, v) ) \right ] r_i(v)
\end{align*}
Our goal will be to bound these terms. The second two are only large when there are not too many edges between vertices in a typical embedding of $T$. To encapsulate this we introduce the following technical definition.
\begin{definition}
We say $\gamma \in \mon(T^i, G)$ is \textit{complete} if $\{\gamma_0, \gamma_j \}, \{\gamma_{a(i + 1)}, \gamma_j \} \in E$ whenever $0 < j < a(i + 1)$ and $x_j$ is not a leaf in $T^{a(i + 1)}$. We say $\gamma$ is non-complete if it is not complete. We let $\Gamma_i^{nc}$ be the set of all non-complete elements of $\mon(T^i, G)$ and let
\begin{align*}
	\Gamma_i^{nc} (u, *) & = \Gamma_i(u, *) \cap \Gamma_i^{nc}\\
	\Gamma_i^{nc} (*, v) & = \Gamma_i( *, v) \cap \Gamma_i^{nc}\\
	\Gamma_i^{nc} (u, v) & = \Gamma_i(u, v) \cap \Gamma_i^{nc}\\
	\Gamma_i^{nc} (p) & = \Gamma_i(p) \cap \Gamma_i^{nc}\\
\end{align*}
\end{definition}

The next lemmas, which are the heart of the proof of Theorem \ref{main}, say that we can write the error terms as weighted sums over non-complete trees. 
\medskip
\begin{lemma} {\label{p1_bound}} For a vertex $v \in V(G)$, let $c_v := d(v)/d$ 
and $f(x) = \log(\frac{xd - i}{d - i}) - \frac{d(x - 1)}{x(d - i)}$. Then
\begin{equation*}
	\Pi_i^1 \geq \sum_{u, v \in V} P\left (\phi^i \in \Gamma_i^{nc}(u, v) \right ) \frac{1}{8}( f(c_u) + f(c_v))
\end{equation*}
and moreover 
\[
\frac{1}{c_{v}} \left [ c_{v} \log \left ( \frac{c_{v} d - i}{d - i} \right ) - (c_{v} - 1) \frac{d}{d - i} \right ]
\]
is non-negative, strictly increasing in $c_v$ when $c_v \geq 1$, and, as long as $d$ is sufficiently large, strictly decreasing in $c_v$ when $i/(d - i)  \leq c_v \leq 1$.
\end{lemma}
\begin{lemma} {\label{p2_bound}}
\begin{equation*}
	\sum_{i = 1}^{t - 1} \Pi_i^2 \geq\sum_{i = 1}^{t-1} \sum_{u, v \in V} P \left (\phi^i \in \Gamma_i^{nc}(u, v) \right )  \frac{1}{8t} \min \left \{ \frac{1}{d(u)}, \frac{1}{d(v)} \right\}
\end{equation*}
with equality only if $\{u, v\} \in E(G)$ whenever $u, v$ have distance less than the diameter of $T$.
\end{lemma}
\begin{lemma} {\label{p3_bound}}
\begin{equation*}
	\Pi_i^3 \leq \sum_{u, v \in V} P \left (\phi^i \in \Gamma_i^{nc}(u, v) \right ) \frac{8 t^2}{d} \left [ \left | \log \left ( \frac{d(u)}{d(v)} \right ) \right | + \frac{8t}{d} \right ]
\end{equation*}
\end{lemma}
The full proof of these lemmas is delayed until the next section but let us discuss very briefly why we ought to be able to write these error terms as weighted sums over non-complete trees in the case that $T$ is a path.

\medskip

Let us start with $\Pi_i^3$: 
\[
P \left ( \phi^i \in \Gamma(v, *) \right ) - P \left (\phi^i \in \Gamma(*, v) \right ) = \sum_{u \in V(G)} [ P \left (\phi^i \in \Gamma(v, u) \right ) - P \left (\phi^i \in \Gamma(u, v) \right ) ]
\]
The quantity $P \left ( \phi^i \in \Gamma(v, u) \right ) - P \left (\phi^i \in \Gamma(u, v) \right )$ is the sum of the measure (with respect to the distribution $\phi^i$) of paths of length $i$ from $u$ to $v$ minus the measure of paths of length $i$ from $v$ to $u$. We get cancellation over complete paths by pairing every complete path with its {\em twist} which is defined as follows.
\begin{definition} Let $p = (u_, w_1, \cdots, w_k, v)$ be a path. The \textit{twist} of $p$, denoted $\tilde{p}$, is the path  $\tilde{p} = (v, w_1, \cdots, w_k, u)$.
\end{definition}
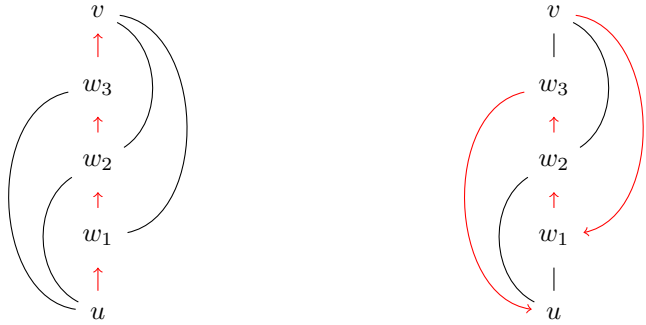
\begin{tikzpicture}
\node[circle] (A) at (0, 0) {$u$};
\node[ circle] (B) at (0, 1) {$w_1$};
\node[ circle] (C) at (0, 2) {$w_2$};
\node[ circle] (D) at (0, 3) {$w_3$};
\node[ circle] (E) at (0, 4) {$v$};

\node[align = flush center, text width = 8cm] (Label) at (0, -1)
{
Path $p = (u, w_1, w_2, w_3, v)$
};

\path[red, ->]  (A) edge (B);
\path[red, ->]  (B) edge (C);
\path[red, ->]  (C) edge (D);
\path[red, ->]  (D) edge (E);

\path  (A) edge[bend left = 60] (C);
\path  (A) edge[bend left = 80] (D);

\path  (C) edge[bend right = 60] (E);
\path  (B) edge[bend right = 80] (E);

\node[circle] (A1) at (6, 0) {$u$};
\node[ circle] (B1) at (6, 1) {$w_1$};
\node[ circle] (C1) at (6, 2) {$w_2$};
\node[ circle] (D1) at (6, 3) {$w_3$};
\node[ circle] (E1) at (6, 4) {$v$};

\node[align = flush center, text width = 8cm] (Label2) at (6, -1)
{
Twist $\tilde p = (v, w_1, w_2, w_3, u)$
};

\path (A1) edge (B1);
\path[red, ->]  (B1) edge (C1);
\path[red, ->]  (C1) edge (D1);
\path  (D1) edge (E1);

\path  (A1) edge[bend left = 60] (C1);
\path[red, ->]  (D1) edge[bend right = 80] (A1);

\path  (C1) edge[bend right = 60] (E1);
\path[red, ->]  (E1) edge[bend left = 80] (B1);
\end{tikzpicture}

A path has the same measure as its twist so this pairing cancels all complete paths.

\medskip

For $\Pi_i^2$, the main idea is that if $p$ is not a complete path, then either $n_+(i) > d(v) - i$ when $\phi^i = p$ or $n_+(i) > d(v) - i$ when $\phi^i = \cev p$. In either case, we get a contribution to $\Sigma_i$.
\medskip

Finally, $\Pi_i^3$ is a technical term that only comes into play when there is unusually large variance in the degrees in $G$. There is no deeper reason for writing it in terms of non-complete paths beyond the fact that it is convenient and Lemma \ref{p1_bound} remains true if we replace $\Gamma^{nc}(u, v)$ with $\Gamma(u, v)$.

\medskip

\subsection{Proof of Lemma \ref{weak}} By Lemmas $\ref{p1_bound}$, $\ref{p2_bound}$, and $\ref{p3_bound}$ we have
\[
H[\phi] \geq \log(n (d)_t) + \sum_{i = 1}^{t - 1} \sum_{u, v \in V(G)} P \left (\phi^i \in \Gamma_i^{nc} (u, v) \right ) ( \Sigma_{u, v, i}^1 + \Sigma_{u, v, i}^2 - \Sigma_{u, v, i}^3 )
\]
where
\begin{align*}
	\Sigma_{u, v, i}^1 & = \frac{1}{8}\left ( \frac{1}{c_v} \left [ c_{v} \log \left ( \frac{c_{v} d - i}{d - i} \right ) - (c_{v} - 1) \frac{d}{d - i} \right ] + \frac{1}{c_{u}} \left [ c_{u} \log \left ( \frac{c_{u} d - i}{d - i} \right ) - (c_{u} - 1) \frac{d}{d - i} \right ] \right ) \\
\Sigma_{u, v, i}^2 & = \frac{1}{8t} \min \left \{ \frac{1}{d(u)}, \frac{1}{d(v)} \right \}\\
\Sigma_{u, v, i}^3 &= \frac{8t^2}{d} \left [ \left | \log \left ( \frac{d(u)}{d(v)} \right ) \right | + \frac{8t}{d} \right ]
\end{align*}
and moreover each term in $\Sigma_{u, v, i}^1$ is non-negative.

%\medskip

%\claim There exists an absolute constant $C_0$ such that if $d \geq C_0 t^8$, then $\Simga_{u, v, i}^1 + \Sigma_{u, v, i}^2 - \Sigma_{u,v, i}^3 \geq 0$

%\begin{proof}

%Let $f(u, v) = \max \{ c_v, 1/c_v, c_u, 1/c_u \} - 1$. We split into three cases based on the value of $f(u, v)$.

%\begin{itemize}

%\item $f(u, v) \leq \frac{1}{1000 t^3}$. Then 
%\[
%\Sigma_{u, v, i} \geq \Sigma_{u, v, i}^2 - \Sigma_{u, v, i}^3 \geq \frac{1}{16dt} - \frac{20t^2}{d} \left ( 2 f(u, v)  + \frac{20t}{d} \geq  \frac{1}{dt} \left ( \frac{1}{16} - \frac{1}{25} \right )  - \frac{400t^3}{d^2} \geq 0
%\]
%\item $ \frac{1}{1000 t^3} \leq f(u, v) \leq 1$. Then WLOG assume that $f(u, v) = \max \{ c_v, 1/c_v \}$. Write $c_v = 1 - x$ for some $- 1 \leq x \leq 1/2$ and let $K : = \frac{d}{d - i} \leq \frac{d }{d - t}$
	%\[
	%\frac{1}{c_v} \left c_v \log \frac{c_v d - i}{d - i} - (c_v - 1) \frac{d}{d - i} \right ] =  \frac{1}{1 - x} \sum_{n = 2}^\infty x^n \left ( \frac{ K^n}{n} - \frac{K^{n + 1}}{n + 1} \right )  \geq \frac{x^2}{24}
	%\]
%\end{itemize}
\begin{claim} {\label{main_claim}} There exists a $d_0(t) \geq 0$ such that if $d \geq d_0(t)$, then 
\[
\Sigma_{u, v, i}^1 + \Sigma_{u, v, i}^2 - \Sigma_{u, v, i}^3 \geq 0
\]

\begin{proof}

Let $\vep := 1/( 2^8 t^3)$. We split into three cases:
\begin{enumerate}
\item $\max \{c_v, c_u, 1/c_u, 1/c_v \} < 1 + \vep$
\item  $\vep + 1 \leq \max \{c_v, c_u, 1/c_u, 1/c_v \} \leq d$
\item $ \max \{c_v, c_u, 1/c_u, 1/c_v \} > d$
\end{enumerate}
First suppose that $ \max \{c_v, c_u, 1/c_u, 1/c_v \} < 1 + \vep$.  Then,
\begin{align*}
\Sigma_{u, v, i}^1 + \Sigma_{u, v, i}^2 - \Sigma_{u, v, i}^3 & \geq \Sigma_{u, v, i}^2 - \Sigma_{u, v, i}^3 \\
& \geq  \frac{1}{8t d(1 + \vep ) }- \frac{8 t^2}{d} \left [ \log \left (  (1 + \vep)^2 \right ) + \frac{8 t}{d} \right ]\\
& \geq   \frac{1}{9 td } - \frac{8^2 t^3}{d^2} - \frac{1 }{16 t d}
\end{align*}
which is positive as long as $d$ is sufficiently large.

\medskip

 Now suppose that $ 1 + \vep  \leq \max \{c_v, c_u, 1/c_u, 1/c_v \} \leq d$. Let 
\[
H_d(x) = \frac{1}{8}\left ( \frac{1}{x} \left [ x \log \left ( \frac{ x d - i}{d - i} \right ) - ( x - 1) \frac{d}{d - i}  \right ] \right )
\]
and let $h_d = \min \{ H_d( 1+ \vep), H_d \left ( 1/(1 + \vep)  \right )$. By Lemma \ref{p1_bound}, $\Sigma_{u, v, i}^1 \geq 2 h_d > 0$. Moreover, since the function $x \log x$ is convex,
\[
\lim_{d \rightarrow \infty}  h_d = \min \left \{  \frac{1}{8}   \left ( \log (1 + \vep)  - \frac{ 1 + \vep - 1}{1 + \vep}  \right ), \frac{1}{8} \left ( \log \left ( \frac{1}{1 + \vep} \right  ) - \frac{ 1/(1 + \vep) - 1}{1/(1 + \vep) } \right ) \right \} > 0
\] 
Therefore there exists a constant $K$ such that for all sufficiently large $d$, $\Sigma_{u, v, i}^1 > K$ and
\begin{align*}
\Sigma_{u, v, i}^1 + \Sigma_{u, v, i}^2 - \Sigma_{u, v, i}^3 & \geq \Sigma_{u, v, i}^1 - \Sigma_{u, v, i}^3 \\
& \geq  K - \frac{8 t^2}{d} \left [ \log \left ( 8d  \right ) + \frac{8 t}{d} \right ] \\
\end{align*}
which is positive as long as $d$ is sufficiently large. 

\medskip

Finally, suppose that $\max \{c_v, c_u, 1/c_u, 1/c_v \} > d$. Since $\delta(G) \geq d/4$, $\max \{c_v, c_u \} \geq d$. Without loss of generality assume that $c_v \geq c_u$. Then 
\[
\log( | d(u)/d(v) | ) = \log( c_v/c_u) \geq \log(c_v/4)
\]
Aditionally, when $x > d$,
\[
H_d(x) \geq  \frac{1}{8}   \left ( \log (x/2) -  \frac{ x- 1}{x} \frac{d}{d - i} \right ) = \frac{1}{8} \log(x) + O(1)
\]
and so
\begin{align*}
\Sigma_{u, v, i}^1 + \Sigma_{u, v, i}^2 - \Sigma_{u, v, i}^3 & \geq \Sigma_{u, v, i}^1 - \Sigma_{u, v, i}^3 \\
& \geq \frac{1}{8} \log (c_v) + O(1)  - \frac{8 t^2}{d} \left [ \log \left ( 4 c_v  \right ) + \frac{8 t}{d} \right ]
\end{align*}
which is positive as long as $d$ is sufficiently large. 

\end{proof}
\end{claim}

By claim \ref{main_claim}
\[
\log( |\mon(T, G)| ) \geq H(\phi) \geq \log( n(d)_t)
\]
Exponentiating proves the inequality in Lemma \ref{weak}.

\medskip

Now suppose that the diameter of $T$ is at least $3$ and $|\mon(T, G)| = n (d)_t$. Then by Claim \ref{p2_bound}, every path of length $3$ in $G$ is a triangle. This implies that $G$ is the disjoint union of cliques. It is easily checked that in this case equality in Lemma $\ref{weak}$ holds iff every clique has the same size.

\section{Proof of Main Lemmas}
In this section we prove the lemmas stated in the previous section. We begin by proving a few other lemmas that we will use to prove Lemmas \ref{p1_bound}, \ref{p2_bound}, and \ref{p3_bound}.
\begin{lemma} {\label{reverse_bound}} If the path from $x_0$ to $x_{a(i + 1)}$ in $T^i$ has length $k$, then for any path $p = (p_0, \cdots, p_k)$ in $G$,
\begin{equation}
	1 - \frac{8 t^2}{d} \leq \frac{P \left (\phi^i \in \Gamma_i(p) \right )}{P \left (\phi^i \in \Gamma_i(\cev p) \right ) } \leq 1 + \frac{8 t^2}{d}
\end{equation}
\end{lemma}
\begin{proof} Let $ \rho = (\rho_0, \cdots, \rho_k)$ be the path in $T^i$ from $x_0$ to $x_{a(i + 1)}$. Let $E_j$ be the event that 
	\begin{itemize}
	\item $\phi_j \not \in p$ if $x_j \not \in \rho$
	\item $\phi_j = p_k$ if $x_j = \rho_k$. 
\end{itemize}
Because $x_0$ is a leaf in $T^i$, $P(E_1 | E_0) = 1/d(p_0)$ and
\begin{equation} {\label{lem11}}
P \left (\phi^i \in \Gamma(p) \right ) = P(\phi_0 = p_0) \prod_{j = 1}^i P(E_j | E_0, \cdots, E_{j - 1} ) = \frac{1}{2|E(G)|} \prod_{j = 2}^i P(E_j | E_0 \cdots E_{j - 1} )
\end{equation}

We proceed by bounding $P(E_j | E_0, \cdots, E_{j - 1})$ for $j \geq 2$. If $x_j \not \in \rho$ then there are at least $\delta(G) - t$ choices for $\phi_j$ that do not intersect $p$ and so,
\begin{equation} {\label{lem12}}
\frac{ \delta(G) - t}{\delta(G)} \leq P(E_j | E_0, \cdots E_{j - 1} ) \leq 1
\end{equation}
If $x_j = \rho_k$ for some $k$ and $E_0, \cdots, E_{j - 1}$ hold, then $\phi_{a(j)} = p_{k - 1}$ and there must be between $d(p_{k-1}) - t$ and $d(p_{k - 1})$ choices for $\phi_j$. So, 
\begin{equation} {\label{lem13}}
\frac{1}{ d(p_{k - 1}) } \leq P(E_j | E_0 \cdots, E_{j - 1} ) \leq \frac{1}{ d(p_{k - 1}) - t}
\end{equation}
Putting (\ref{lem11}), (\ref{lem12}), and (\ref{lem13}) together gives
\[
\frac{1}{2|E(G)|} \left ( \frac{ \delta(G) - t}{\delta(G)} \right ) ^{t - r} \prod_{j = 1}^{k-1} \frac{1}{d(p_j)} \leq P \left (\phi^i \in \Gamma_i (p) \right ) \leq \frac{1}{2|E(G)|} \prod_{j = 1}^{k-1} \frac{1}{d(p_j) - t}
\]
Since $\{p_1, \cdots, p_{k - 1} \} = \{ \cev{p_1}, \cdots, \cev{p_{k - 1}} \}$ as sets, applying the above to $p$ and $\cev p$ and using $\delta(G) \geq d/4$ gives
\[
1 - \frac{8t^2}{d} \leq \left ( \frac{ \delta(G) - t}{ \delta(G)} \right )^t \leq \frac{ P \left (\phi^i \in \Gamma_i(p) \right )}{P \left (\phi^i \in \Gamma_i(\cev p) \right )} \leq \frac{1}{\left ( \delta(G) - t \right ) / \left ( {\delta(G) } \right ) ^t } \leq 1 + \frac{8 t^2}{d}
\]
as long as $d$ is sufficiently large.
\end{proof}

\begin{corollary} {\label{cor}} Let $u, v \in V(G)$, as long as $\Gamma_i(u, v)$ is non-empty, 
	\begin{equation}
		1 - \frac{8 t^2}{d} \leq \frac{P \left  (\phi^i \in \Gamma_i(u,v) \right )} {P \left  (\phi^i \in \Gamma_i(v, u)  \right )} \leq 1 + \frac{8t^2}{d} {\label{cor_1}}
\end{equation}
and
\begin{equation}
	1 - \frac{8t^2}{d} \leq \frac{P \left (\phi^i \in \Gamma_i(u, *) \right )} {P\left (\phi^i \in \Gamma_i(* , u) \right )} \leq 1 + \frac{8t^2}{d} {\label{cor_2}}
\end{equation}
\end{corollary}
\begin{proof} Let $k$ be the length of the path from $x_0$ to $x_{a(i + 1)}$ in $T^i$ and let $P_k(u, v)$ denote the set of paths of length $k$ from $u$ to $v$ in $G$. Then
\[
\frac{P \left (\phi^i \in \Gamma_i(u, v) \right )}{P\left (\phi^i \in \Gamma_i(v, u) \right )} = \left ( \sum_{p \in P_k(u, v) } P\left (\phi^i \in \Gamma_i(p) \right ) \right )  \Big /  \left ( { \sum_{p \in P_k(u, v) } P \left (\phi^i \in \Gamma_i(\cev p) \right )} \right ) 
\]
and (\ref{cor_1}) follows from Lemma \ref{reverse_bound}. Similarly,

\begin{align*}
\frac{P \left (\phi^i \in \Gamma_i(u, *) \right )}{P\left (\phi^i \in \Gamma_i(*, u) \right )} & = 
\left ( \sum_{v \in V(G)} \sum_{p \in P_k(u, v) } P \left (\phi^i \in \Gamma_i(p) \right ) \right )  \Big / \left ( \sum_{v \in V(G)} \sum_{p \in P_k(v, u) } P \left (\phi^i \in \Gamma_i(p) \right ) \right ) \\
& =  \left ( \sum_{v \in V(G)} \sum_{p \in P_k(u, v)}  P\left (\phi^i \in \Gamma_i(p)  \right ) \right ) \Big / \left ( \sum_{v \in V(G)} \sum_{p \in P_k(u, v)} P \left (\phi^i \in \Gamma_i(\cev p) \right )  \right ) 
\end{align*}

and (\ref{cor_2}) follows from Lemma \ref{reverse_bound}
\end{proof}

\begin{definition} For a path $p$, we define $\Gamma_i^{\overline{nc} }$ to be the set of functions in $\mon (T^{a(i + 1)}, G)$ which are restrictions of functions in $\Gamma_i^{nc}(p)$ to the tree $T^{a(i + 1)}$.
\end{definition}

\begin{lemma} {\label{reverse_bound_2}} Let $u, v \in V(G)$ such that $\Gamma_i^{nc} (u, v)$ is non-empty. Then
\[
1 - \frac{8t^2}{d} \leq \frac{ P \left (\phi^i \in \Gamma_i^{nc}(u, v) \right ) }{ P\left (\phi^i \in \Gamma_i^{nc} (v, u) \right )} \leq 1 + \frac{8 t^2}{d}
\]
\end{lemma}
\begin{proof} Let $k$ be the length of the path from $x_0$ to $x_{a(i + 1)}$ in $T$. Say a path $p = (p_0, \cdots, p_k)$ from $u$ to $v$ in $G$ is complete if $\{p_0, p_j\}, \{p_k, p_j\} \in E(G)$ for every $0 < j < k$. If $p$ is complete, define it's twist to be the path $\tilde{p} = (p_k, p_1, \cdots, p_{k-1}, p_0)$. Let $P_k^c(u, v)$ be the set of all complete paths of length $k$ from $u$ to $v$ and let $P_k^{nc}(u, v)$ be the set of all non-complete paths of length $k$ from $u$ to $v$. If $\phi \in \mon(T^i, G)$ then $\phi^i \in \Gamma_i^{nc}(p)$ iff $\phi^{a(i + 1)} \in \Gamma_i^{\overline{nc}} (p)$ and so 
	\[
	P \left (\phi^i \in \Gamma_i^{nc} (p) \right ) = P\left (\phi^{a(i + 1)} \in \Gamma_i^{\overline{nc}} (p) \right )
	\]
Using the two facts 
\begin{itemize}
\item If $p$ is a non-complete path then $\Gamma_i^{nc} (p) = \Gamma_i(p)$ 
\item The twist is a bijection from $P_k^c (u, v)$ to $P_k^c(v, u)$ for any vertices $u, v$
\end{itemize}
we obtain
\begin{align}
	P \left (\phi^i \in \Gamma_i^{nc} (u, v) \right ) = & \sum_{p \in P_k^{nc}(u, v)} P \left (\phi^i \in \Gamma_i^{nc} (p)  \right ) + \sum_{p \in P_k^{c}(u, v) } P \left (\phi^i \in \Gamma_i^{nc} (p) \right ) \nonumber \\
	= & \sum_{p \in P_k^{nc}(u, v) } P \left (\phi^i \in \Gamma_i(p) \right ) + \frac{1}{2} \sum_{p \in P_k^{c}(u, v) }  P \left ( \phi^{a(i + 1)} \in \Gamma_i^{\overline{nc}}(p)  \right ) \nonumber \\
 & +  \frac{1}{2}  \sum_{p \in P_k^c(u, v)} P \left (\phi^{a(i + 1)} \in \Gamma_i^{\overline{nc}} \left (\widetilde{ \cev p } \right ) \right ) {\label{lem21}}
\end{align}
On the other hand, since the map $p \rightarrow \cev p$ is a bijection from $P_k^{nc}(u, v)$ to $P_k^{nc}(v, u)$ and also from $P_k^c (u, v)$ to $P_k^c(v, u)$, 
\begin{align}
	 P\left (\phi^i \in \Gamma_i^{nc} (v, u) \right ) = & \sum_{p \in P_k^{nc}(v, u)} P\left (\phi^i \in \Gamma_i^{nc} (p) \right ) + \sum_{p \in P_k^{c}(v, u) } P\left (\phi^i \in \Gamma_i^{nc} (p) \right ) \nonumber \\
	 = & \sum_{p \in P_k^{nc}(u, v)} P\left (\phi^i \in \Gamma_i(\cev p) \right ) + \frac{1}{2} \sum_{p \in P_k^c(u, v)} P\left (\phi^{a(i + 1))} \in \Gamma_i^{\overline{nc}} (\cev p) \right ) \nonumber \\
& +  \frac{1}{2} \sum_{p \in P_k^c(u, v)} P\left (\phi^{a(i + 1)} \in \Gamma_i^{\overline{nc}} (\tilde{p}) \right ) {\label{lem22}}
\end{align}
By (\ref{lem21}) and (\ref{lem22}), it suffices to show the following three inequalities
\begin{itemize}
	\item For $p \in P_k^{nc}(u, v)$, 
		\begin{equation}
			1 - \frac{8t^2}{d} \leq \frac{ P \left (\phi^i \in \Gamma_i(p) \right )}{ P\left (\phi^i \in \Gamma_i(\cev p) \right )} \leq 1 + \frac{8t^2}{d} {\label{p_over_pr}}
		\end{equation}
	\item For $p \in P_k^c (u, v)$,  
		\begin{equation}
			1 - \frac{8 t^2}{d} \leq \frac{P\left (\phi^{a(i + 1)} \in \Gamma_i^{\overline{nc}}(p) \right )}{P\left (\phi^{a(i + 1)} \in \Gamma_i^{\overline{nc}} (\tilde{p}) \right )} \leq 1 + \frac{8t^2}{d} \label{twisted_ineq}
	\end{equation}
	\item For $p \in P_k^c (u, v)$,  
		\begin{equation}
			1 - \frac{8t^2}{d} \leq \frac{P \left (\phi^{a(i + 1)} \in \Gamma_i^{\overline{nc}} \left (\widetilde{ \cev p  } \right ) \right )}{P \left (\phi^{a(i + 1)} \in \Gamma_i^{\overline{nc}} (\cev p) \right )} \leq 1 + \frac{8t^2}{d} {\label{twisted_ineq2}}
	\end{equation}
\end{itemize}
(\ref{p_over_pr}) is exactly Lemma $\ref{reverse_bound}$. (\ref{twisted_ineq}) and (\ref{twisted_ineq2}) are equivalent after replacing $p$ with $ \widetilde{ \cev p }$, so we will proceed by proving (\ref{twisted_ineq}).

\medskip

{If $p$ is a complete path and $\gamma = (\gamma_0, \cdots, \gamma_{a(i + 1)} ) \in \Gamma_i^{\overline{nc}} (p)$, then (see the definition) the twist of $\gamma$, 
\[ \tilde{\gamma} = (\gamma_{a(i + 1)}, \gamma_1, \cdots, \gamma_{a(i + 1) - 1}, \gamma_0 ).\]  The twist is a bijection between $ \Gamma_i^{\overline{nc}}(p)$ and $\Gamma_i^{\overline{nc}} (\tilde{p})$ so
\[
\frac{P\left (\phi^{a(i + 1)} \in \Gamma_i^{\overline{nc}} (p) \right ) }{ P\left (\phi^{a(i + 1)} \in \Gamma_i^{\overline{nc}} ( \tilde{p}) \right )} =  \left ( \sum_{ \gamma \in \Gamma_i^{\overline{nc}}(p)} P\left ( \phi^{a(i + 1)} = \gamma \right )  \right ) \Big /  \left ( \sum_{ \gamma \in \Gamma_i^{\overline{nc}} (p)} P \left (\phi^{a(i + 1)} = \tilde{\gamma}  \right ) \right )
\]}
Therefore it suffices to show that for any $\gamma \in \Gamma_i^{\overline{nc}}(p)$,
\begin{equation} {\label{lem22_final}}
1 - \frac{8t^2}{d} \leq \frac{P\left (\phi^{a(i + 1)}  = \gamma \right )}{P \left (\phi^{a(i + 1)} = \tilde{\gamma} \right ) } \leq 1 + \frac{8t^2}{d}
\end{equation}
By the construction of our probability distribution $\phi$,
\begin{align*}
\frac{1}{2|E(G)| }  \prod_{j = 2}^{a(i + 1)}  \frac{1}{ d( \gamma_{a(j)}) } & \leq P(\phi^{a(i + 1)} = \gamma)  \leq \frac{1}{2|E(G)|} \prod_{j = 2}^{a(i + 1)} \frac{1}{ d(\gamma_{a(j)}) - t} \\
\frac{1}{2|E(G)| }  \prod_{j = 2}^{a(i + 1)}  \frac{1}{ d( \gamma_{a(j)}) } & \leq P(\phi^{a(i + 1)} = \tilde{\gamma})  \leq \frac{1}{2|E(G)|} \prod_{j = 2}^{a(i + 1)} \frac{1}{ d(\gamma_{a(j)}) - t}
\end{align*}
Using $\delta(G) \geq d/4$ gives (\ref{lem22_final}).
\end{proof}
\medskip

\begin{lemma} {\label{twist}} For any $0 \leq i < t$ and $u, v \in V(G)$,
\[
P\left (\phi^i \in \Gamma_i(v, u) \right ) - P\left (\phi^i \in \Gamma_i(u, v) \right )  = P\left (\phi^i \in \Gamma_i^{nc}(v,u) \right ) - P\left (\phi^i \in \Gamma_i^{nc}(u, v) \right )
\]
\end{lemma}
\begin{proof} Let $\Gamma_{i}^{c}(v, u) = \Gamma_{i}(v, u) \setminus \Gamma_{i}^{nc}(v,u)$. We need to show
\[
P\left (\phi^i \in \Gamma_i^c(v, u) \right ) - P\left (\phi^i \in \Gamma_i^c(u, v) \right ) = 0 
\]
Let $\Gamma_i^{ \overline c}(v, u)$ be the set of functions in $\mon(T^{a(i + 1)}, G)$ which are restrictions of functions in $\Gamma_i^c(v, u)$ to $T^{a(i + 1)}$. 

\medskip

For $\gamma \in \Gamma_{i}^{ \overline c} (v, u)$ define the twist of $\gamma$ by $\tilde{\gamma} = (\gamma_{a(i + 1)}, \gamma_1, \cdots, \gamma_{a(i + 1) - 1}, \gamma_0)$. The twist is a bijection from $\Gamma_{i}^{ \overline c}(v, u)$ to $\Gamma_{i}^{ \overline c}(u, v)$ and for $\gamma \in \Gamma_i^{\overline c} (v, u)$,
\[
P \left (\phi^{a(i + 1)} = \tilde{\gamma} \right ) = \frac{1}{2|E(G)|} \prod_{j = 2}^{a(i + 1)} \frac{1}{d(\gamma_{a(j)}) - 1 - |N(\gamma_{a(j)}) \cap \{\gamma_1, \cdots, \gamma_{j - 1} \} | } = P \left (\phi^{a(i + 1)} = \gamma \right )
\]
Therefore,
\begin{align*}
	P \left (\phi^i \in \Gamma_i^c(v, u) \right ) - P\left (\phi^i \in \Gamma_i^c(u, v) \right ) &= P\left (\phi^{a(i + 1)} \in \Gamma_{i}^{\overline c}(v, u) \right ) - P\left (\phi^{a(i + 1)} \in \Gamma_{i}^{\overline c} (v, u) \right)\\
	&= \sum_{\gamma \in \Gamma_{i}^{ \overline c} (v, u) } P\left (\phi^{a(i + 1)} = \gamma \right ) - p \left (\phi^{a(i + 1)} = \tilde{\gamma} \right )\\
	&= 0
\end{align*}
\end{proof}

We are now ready to prove Lemmas \ref{p1_bound}, \ref{p2_bound}, and \ref{p3_bound}.
\begin{proof}[Proof of Lemma \ref{p3_bound}]
\begin{align}
\Pi_i^3 &= \sum_{v \in V(G)} [ P(\phi^i \in \Gamma_i(v, *)) - P(\phi^i \in \Gamma_i(*, v)) ] r_i(v) \nonumber \\
&= \frac{1}{2} \sum_{u,v \in V(G)} [P(\phi^i \in \Gamma_i(v, u)) - P(\phi^i \in \Gamma_i(u, v))] [r_i(v) - r_i(u)] \nonumber \\
&= \frac{1}{2} \sum_{u , v \in V(G)} [ P(\phi^i \in \Gamma_i^{nc} (v, u)) - P(\phi^i \in \Gamma_i^{nc}(v, u)) ] [r_i(v) - r_i(u)] {\label{twist_use}} \\
& \leq \frac{1}{2} \sum_{v, u \in V(G)} \sup_{0 \leq j, k \leq t} [P(\phi^i \in \Gamma_i^{nc}(v, u)) - P(\phi^i \in \Gamma_i^{nc}(u, v)) ] \left [ \log \left ( \frac{ d(v) - j}{ d(u) - k} \right ) \right ] \nonumber \\
& \leq \frac{1}{2} \sum_{u, v \in V(G)} |P(\phi^i \in \Gamma_i^{nc}(v, u)) - P(\phi^i \in \Gamma_i^{nc}(u, v)) | \left [ \left | \log \left ( \frac{ d(v)}{ d(u)} \right ) \right | + \frac{8t}{d} \right ] \nonumber \\
& \leq \sum_{u, v \in V(G)} \frac{8t^2}{d} P(\phi^i \in \Gamma_i^{nc}(v, u)) \left [ \left | \log \left ( \frac{ d(v)}{ d(u)} \right ) \right | + \frac{8t}{d} \right ] {\label{reverse_use}}
\end{align}
(\ref{twist_use}) is Lemma \ref{twist} and (\ref{reverse_use}) follows from Lemma \ref{reverse_bound_2}.
\end{proof}

\begin{proof}[Proof of Lemma \ref{p2_bound}] 
	For $u, v \in V(G)$, let $\Gamma_i(*, u, v)$ be the set of elements $\gamma \in \mon(T^i, G)$ such that $\gamma_{a(i + 1)} = v$ and $u$ is in $\im(\gamma)$. Let $\Gamma_i(v, u, *)$ be the set of $\gamma \in \mon(T^i, G)$ such that $\gamma_0 = v$, $u$ is in $\im(\gamma)$ \textit{and} if $\gamma_j = u$, then $x_j$ has an ancestor in $T^j$. (For technical reasons there is an asymmetry in the definition between $\Gamma_i( *, u, v)$ and $\Gamma_i(v, u, *)$).
\begin{align}
	\sum_{i = 1}^{t-1}\Pi_i^2 &= \sum_{i = 1}^{t - 1} \sum_{v \in V} \frac{d(v)}{2|E(G)|} (r_i(v) - \log (d(v) - i)) \nonumber \\
	&= \sum_{ i = 1}^{t-1} \sum_{ u \in V(G) } \frac{ P( \phi^i \in \Gamma_i(u, *) )} { P( \phi^i \in \Gamma_i (*, u))} \sum_{\gamma \in \Gamma(*, u) } P(\phi^i = \gamma) \log \left ( \frac{d(u) - | N(u) \cap \{\gamma_0, \cdots, \gamma_{i - 1} \} | }{d(u) - i} \right ) \nonumber \\
	&\geq \frac{1}{2} \sum_{ i = 1}^{t-1} \sum_{ u \in V(G) } \sum_{\gamma \in \Gamma(*, u) } P(\phi^i = \gamma) \log \left ( \frac{d(u) - | N(u) \cap \{\gamma_0, \cdots, \gamma_{i - 1} \} | }{d(u) - i} \right ) {\label{first}} \\
	& \geq \frac{1}{2} \sum_{ i = 1}^{t-1} \sum_{u \in V(G) }  \sum_{\gamma \in \Gamma(*, u) } P(\phi^i = \gamma) | \{ j : j \neq a(i + 1), \{\gamma_j, \gamma_{a(i + 1)} \} \not \in E(G) \} | \log \left ( \frac{d(u) + 1}{d(u)} \right ) {\label{end}}
\end{align}
(\ref{first}) follows from Corollary \ref{cor}. (\ref{end}) counts weighted pairs of embedded trees and missing edges. In the following we write $\sum_{(u, v) \not \in E(G) }$ to mean we are summing over ordered pairs $(u, v)$ such that $\{u, v \} \not \in E(G)$. We can sum over the missing edges and obtain that this is equal to
\begin{align}
	& \frac{1}{2} \sum_{i = 1}^{t - 1} \sum_{(v, u) \not \in E(G)}  \sum_{\gamma \in \Gamma_i (*, v,u) } P( \phi^i = \gamma) \log \left ( \frac{ d(u) + 1}{ d(u)} \right ) \nonumber \\
	& \geq \frac{1}{8} \sum_{i = 1}^{t-1} \sum_{(v, u) \not \in E(G)} [P( \phi^i \in \Gamma_i(*, v, u) ) + P(\phi^i \in \Gamma_i(u, v) )] \log \left ( \frac{d(u) + 1}{ d(u)} \right ) {\label{p2_bound_2}}\\
	& = \frac{1}{8} \sum_{i = 1}^{t-1} \sum_{(v, u) \not \in E(G)} \left [ P( \phi^i \in \Gamma_i( *, v, u)) + \sum_{j = 1}^i \frac{1}{t - j} \sum_{ \gamma \in \Gamma_j(u, v)} P(\phi^j  = \gamma)  \right ] \log \left ( \frac{d(u) + 1}{ d(u)} \right ) \nonumber \\
	& \geq \frac{1}{8t} \sum_{i = 1}^{t-1} \sum_{(v, u) \not \in E(G)} \left ( P( \phi^i \in \Gamma_i( *, v, u)) + \sum_{j = 1}^i \sum_{ \gamma \in \Gamma_j(u, v)} P(\phi^j = \gamma) \right ) \log \left ( \frac{d(u) + 1}{ d(u)} \right ) {\label{end2}}
\end{align}
(\ref{p2_bound_2}) follows from Corollary \ref{cor} and the fact that $P(\phi^i \in \Gamma_i(v, u)) \leq P(\phi^i \in \Gamma_i(*, v, u) )$. Every element in $\Gamma_i(u, v, *)$ restricts to an element in in $\Gamma_j(u, v)$ for some $j$ so (\ref{end2}) is at least
\begin{align}
	& \frac{1}{8t} \sum_{i = 1}^{t-1} \sum_{(v, u) \not \in E(G)} \left ( P( \phi^i \in \Gamma_i( *, v, u) + P( \phi^i \in \Gamma_i(u, v, *) \right ) \log \left ( \frac{d(u) + 1}{ d(u)} \right ) {\label{second_to_last_step}} 
\end{align}
Now notice that if $\phi^i \in \Gamma_i^{nc}(v, u)$, then there is an $x_j \in T^{a(i + 1)}$ which is not a leaf such that either $(v, \phi_j) \not \in E(G)$ or $(\phi_j, u) \not \in E(G)$. In the former case $\phi^i \in \Gamma_i(*, \phi_j, u)$ and in the latter case $\phi^i \in \Gamma_i(v, \phi_j, *)$. So we conclude that (\ref{second_to_last_step}) it at least
\begin{align}
	&\frac{1}{8t} \sum_{i = 1}^{t-1} \sum_{u, v \in V(G)} P( \phi^i \in \Gamma_i^{nc} (v, u)) \min \left \{ \log \left ( \frac{d(u) + 1}{ d(u)} \right ) , \log \left ( \frac{ d(v) + 1}{ d(v)} \right ) \right \} \nonumber \\
	&\geq \frac{1}{16t} \sum_{i = 1}^{t-1} \sum_{u, v \in V(G)} P( \phi^i \in \Gamma_i^{nc} (v, u)) \min \left \{ \frac{1}{ d(u) }, \frac{1}{d(v)} \right \} \nonumber
\end{align}
\end{proof}

\begin{proof}[Proof of Lemma \ref{p1_bound}] For $v \in V(G)$, let $L(v) =  c_{v} \log \left ( \frac{c_{v} d - i}{d - i} \right ) - (c_{v} - 1) \frac{d}{d - i}$.  
Then by Lemma \ref{jensens_bound} 
\begin{align*}
	\Pi_{i}^3 &= \sum_{v \in V} \frac{d(v)}{2|E(G)|} \log( d(v) - i) ) - \log(d - i)\\
	&= \frac{1}{n} \sum_{v \in V} L(v)\\
	& = \frac{1}{n} \sum_{u, v \in V(G) } \frac{2 |E|}{ d(v)} P(\phi^i \in \Gamma_i(v, u)) L(v)\\
	& = \frac{1}{2} \sum_{u, v \in V(G) } P(\phi^i \in \Gamma_i(v, u) ) \frac{L(v)}{c_{v}}\\
	& = \frac{1}{4} \sum_{u, v \in V(G) } P(\phi^i \in \Gamma_i(v, u)) \frac{L(v)}{c_{v}}+ P( \phi^i \in \Gamma_i (u, v)) \frac{L(u)}{c_{u}}  \\
	& \geq \frac{1}{8} \sum_{u, v \in V(G) } P(\phi^i \in \Gamma_i(v, u)) \left ( \frac{L(v)}{c_{v}}  + \frac{L(u)}{c_{u}} \right )\\
	& \geq \frac{1}{8} \sum_{u,v \in V(G) } P(\phi^i \in \Gamma_i^{nc}(v, u)) \left ( \frac{L(v)}{c_{v}}  + \frac{L(u)}{c_u} \right )
\end{align*}
Each term is non-negative because the function $x \log (x - k)$ is convex when $x \geq 2k$. Similary, by convexity, $L(v)/c_v$ is increasing when $c_v \geq 1$ and decreasing when $i/(d - i) \leq c_v \leq 1$.
\end{proof}
\section{Concluding Remarks}
\begin{itemize}
\item Our proof technique of Theorem \ref{main} can obtain a bound of $d_0(t) = O(t^4)$.
\item If $T$ is a tree and has parts of size $t_1, t_2$ as a bipartite graph, then if $G$ is an $n$-vertex bipartite graph with $e$ edges and parts of size $a$ and $b$, then using the proof technique of Theorem \ref{main}, we obtain the bound
	\[
	|\mon(T, G)| \geq a  \left ( \frac{e}{a} \right )_{t_2} \left ( \frac{e}{b} \right )_{t_1 - 1} + b \left ( \frac{e}{a} \right ) _{t_1 - 1} \left (\frac{e}{b} \right )_{t_2} \]
which extends the results of Hoory~\cite{Hoory} for non-backtracking walks in bipartite graphs.
\item One could ask the same question with forests replaced by trees. In this case a disjoint union of cliques is no longer the minimizer. For example if $F$ is the disjoint union of two edges, then if $G$ is the disjoint union of $n/(d + 1)$ cliques of size $d + 1$, there are about $n^2 d^2$ embedding of $F$ into $G$. On the other hand the average degree of $K_{d/2, n - d/2}$ is approximatly $d$ but there are only about $n^2 d (d - 2)$ embeddings of $F$ into $K_{d/2, n - d/2}$. In fact if $F$ is the disjoint union of $k$ edges, then one can show by induction and smoothing that the graph $G$ with vertex set $A \cup B$ and $ \{u, v \} \in E(G)$ iff $u \in A$ or $v \in A$ is a minimizer of $|\mon(F, G')|$ over all graphs $G'$ with $|V(G')| = |V(G)|$ and $|E(G')| = |E(G)|$. Moreover if $G$ is an $n$-vertex graph with average degree $d$, then $|\mon(F, G)| \geq (1 - o(1)) n(n -2) \cdots (n - 2k + 2) d(d - 2) \cdots (d - 2k + 2)$.

\end{itemize}

\section{Appendix}
Here is a short argument that gets a bound 
\[
|\mon(T, G)| \geq n (d - t(t - 1))^t = \left ( 1 - O \left ( \frac{t^3 }{d} \right ) \right ) n (d)_t
\]
\begin{proof}

For simplicity assume that $T$ is a path. Let us imagine picking a vertex and taking a walk of length $t$. Instead of requiring our walk to be a path i.e. not intersecting itself, an adversary will prevent us from taking certain steps; they will try to minimize the number of walks we can take by - at step $i$ - picking at most $i$ neighbors of our current vertex and forbidding us from traveling to any of them during our next move.

\medskip

On the one hand, our adversary can restrict us to only taking paths by forbidding us from traveling to vertices we have already visited. On the other hand, the adversary's optimal strategy does not need to depend on the history of our walk - indeed at step $i$ they are simply trying to minimize the number of walks of length $t - i$ we can take from that point onward. So for every vertex $v$ and $0 < i < t$, we have a set $F_{v, i}$ of neighbors our adversary would forbid if we were at $v$ at step $i$ in our walk where $|F_{v, i}| = i$. Let $G' = (V(G), E')$ where $E' = E(G) - \{ \{ v , w \} : v \in V, w \in F_{v,1} \cup \cdots \cup F_{v, t - 1} \}$. Then the number of walks we can make, is at least the total number of walks in $G'$ which has average degree at least $ d - t(t - 1)$. So, by a well known lower bound on the number of walks in a graph \cite{Sidorenko}, $| \mon(T, G)| \geq n (d - t(t - 1))^t$
\end{proof}

\begin{lemma} [Error Bound on Jensen's Inequality] {\label{jensens_bound}} Let $d_1, \cdots, d_n \geq k$ and $E[d_i] = d$. Let $c_i = d_i/d$. Then
\[
\left ( \frac{1}{d} \sum_{i = 1}^n \frac{d_i}{n} \log(d_i - k) \right ) - \log(d - k) = \frac{1}{n} \sum_{i = 1}^n \left [ c_i \log \left ( \frac{ c_i d - k}{d - k} \right ) - (c_i - 1) \left ( \frac{d}{d - k} \right ) \right ] \\
\]
\end{lemma} 
\begin{proof} In general for a differentiable function $\varphi$ and $x$ a random variable, letting $y = E[x]$,
	\[
	E[ \varphi(x)] - \varphi(E[x]) = E[\varphi(x)] - E[  \varphi(E[x]) - \varphi'(y) (x - E[x] ) ] = E[ \varphi(x) - \varphi(E[x]) - \varphi'(y) (x - E[x]) ]
	\]
	Applying this to $x = d_i$ and $\varphi(x) = x \log(x - k)$ gives
	\begin{align*}
			\frac{1}{d} \sum_{i = 1}^n \frac{d_i}{n} \log(d_i - k) - \log(d - k) &= \frac{1}{dn} \sum_{i = 1}^n c_i d \log( c_i d - k) - d \log(d - k) - \left ( \log(d - k) + \frac{d}{d - k} \right ) (c_i d - d)\\
				& = \frac{1}{dn} \sum_{i = 1}^n d \left [ c_i \log \left ( \frac{ c_i d - k}{d - k} \right ) - (c_i - 1) \left ( \frac{d}{d - k} \right ) \right ] \\
					& = \frac{1}{n} \sum_{i = 1}^n \left [ c_i \log \left ( \frac{ c_i d - k}{d - k} \right ) - (c_i - 1) \left ( \frac{d}{d - k} \right ) \right ] \\
	\end{align*}
\end{proof}

\begin{proof} [Proof that Lemma \ref{weak} implies Theorem \ref{main} ]
We assume that $t \geq 3$ as otherwise Theorem \ref{main} is trivial. If any vertex $v \in V(G)$ has degree less than $d/4$ then we will delete it. We observe that, letting $G'$ be the resulting graph,
\begin{itemize}
\item $ \mon(T, G') \subseteq \mon(T, G)$
\item $G'$ has $n - 1$ vertices and average degree $d' \geq d$
\item $ (n - 1) (d')_t \geq n (d)_t$
\end{itemize}
We continue deleting vertices in this way until we reach a graph $\hat{G}$ with $\hat{n}$ vertices such that the average degree $\hat d$ satisfies $\delta( \hat G) \geq \hat d/4$. At this point we may apply Lemma \ref{weak} to $\hat{G}$ to conclude there are at least $\hat{n} (\hat{d})_t$ copies of $T$ in $\hat G$. Since $\mon(T, \hat{G} ) \subseteq \mon(T, G)$, we have $|\mon(T,G) | \geq |\mon( T, \hat G) | \geq \hat{n} (\hat d)_t \geq n (d)_t$.
\end{proof}

\section{Acknowledgments}
We would like to thank Jaques Verstraete for helpful conversations and comments on this paper.

\end{document}